\documentclass[11pt]{article}

\usepackage[a4paper,margin=1.15in]{geometry}
\usepackage{amsmath,amssymb,amsthm,mathtools}
\usepackage{bm}
\usepackage{booktabs}
\usepackage{enumitem}
\usepackage{hyperref}

\hypersetup{
	colorlinks=true,
	linkcolor=blue,
	citecolor=blue,
	urlcolor=blue
}

\numberwithin{equation}{section}

\newtheorem{theorem}{Theorem}[section]
\newtheorem{proposition}[theorem]{Proposition}
\newtheorem{lemma}[theorem]{Lemma}

\newtheorem{definition}[theorem]{Definition}
\newtheorem{remark}[theorem]{Remark}
\newtheorem{example}[theorem]{Example}

\newcommand{\R}{\mathbb R}
\newcommand{\Sd}{\mathbb S^d}
\newcommand{\Spd}{\mathbb S^d_+}
\newcommand{\co}{\operatorname{co}}
\newcommand{\cl}{\operatorname{cl}}
\newcommand{\tr}{\operatorname{tr}}
\newcommand{\dist}{\operatorname{dist}}

\newcommand{\Hess}{\mathcal H}
\newcommand{\LimHess}{\mathcal L}
\newcommand{\LOC}{\operatorname{LOC}}
\newcommand{\NLOC}{\operatorname{NLOC}}
\newcommand{\CONV}{\operatorname{CONV}}

\title{Local Nonconvexity Indices for \(C^{1,1}\) Functions via Generalized Hessians}

\author{Marina Palaisti}

\date{\today}

\begin{document}
	
	\maketitle
	
	\begin{abstract}
		Davydov, Moldavskaya, and Zitikis introduced local indices for quantifying the lack of convexity of a \(C^2\) function by measuring the nuclear-norm distance of its Hessian from the cone of positive semidefinite matrices. This paper develops a local analogue for functions of class \(C^{1,1}\). At a point \(x\), the classical Hessian is replaced by the Clarke-type generalized Hessian set \(\Hess(h;x)\), defined as the closed convex hull of limiting Hessians at nearby twice differentiability points. Evaluating the same spectral functional over \(\Hess(h;x)\) gives an interval-valued local nonconvexity index whose lower and upper endpoints represent, respectively, the least and greatest visible second-order nonconvexity at \(x\). The construction reduces to the original smooth index when \(h\in C^2\), vanishes for convex \(C^{1,1}\) functions, is invariant under orthogonal changes of variables, satisfies a subadditivity inequality for the upper endpoint under sums, and is upper semicontinuous in its upper endpoint. We also relate the upper endpoint to a pointwise weak-convexity curvature modulus and give explicit \(C^{1,1}\setminus C^2\) examples. The paper is deliberately local in scope: it proposes a scalar diagnostic extracted from generalized Hessian sets, not a replacement for the richer second-order variational theory of nonsmooth convexity.
	\end{abstract}
	
	
	
	\section{Introduction}
	
	Convexity is one of the central structural assumptions in optimization, variational analysis, and risk modelling. It provides qualitative stability, global optimality of local minimizers, and strong algorithmic guarantees. Many functions arising in applications, however, are nonconvex, and many are not twice continuously differentiable. This motivates quantitative tools that can measure not only whether convexity fails, but also how severely it fails at a given point.
	
	Davydov, Moldavskaya, and Zitikis \cite{davydov2019} proposed a local index of nonconvexity for \(C^2\) functions. If \(h:G\to\R\) is twice continuously differentiable on a domain \(G\subset\R^d\), their index evaluates the negative spectral part of the Hessian \(D^2h(x)\). Equivalently, it measures the nuclear-norm distance from \(D^2h(x)\) to the cone of positive semidefinite matrices. This interpretation is attractive because it is local, spectral, and geometrically transparent.
	
	The limitation is that the construction uses the classical Hessian. It therefore does not directly apply to functions that are differentiable with Lipschitz gradient but fail to be \(C^2\). Such functions are common in nonsmooth optimization and variational analysis. For \(C^{1,1}\) functions, however, second-order information is still available through generalized Hessian sets. By Rademacher's theorem, the gradient is differentiable almost everywhere, and one can form a compact convex set of limiting Hessian matrices at each point. This is the setting of the present note.
	
	The main idea is simple. We replace the single Hessian matrix \(D^2h(x)\) in the smooth index by the generalized Hessian set \(\Hess(h;x)\). Since \(\Hess(h;x)\) is a set rather than a single matrix, the local index becomes interval-valued. Its lower endpoint records the smallest value of the Davydov--Moldavskaya--Zitikis spectral functional visible in the generalized Hessian set, while its upper endpoint records the largest. Thus the proposed object is not merely a number; it is the interval
	\[
	\mathcal I_h(x)
	:=
	\left[
	\underline{\LOC}(h;x),
	\overline{\LOC}(h;x)
	\right].
	\]
	The interval collapses to a singleton in the \(C^2\) case.
	
	The scope of this contribution is intentionally modest. Recent nonsmooth optimization theory has developed powerful second-order frameworks involving prox-regularity, variational convexity, second-order subdifferentials, generalized twice differentiability, quadratic bundles, and generalized Newton methods; see, for example, \cite{rockafellarwets1998,poliquinrockafellar1996,lewiszhang2013,khanh2023,gfrerer2025,khanh2025,khanh2024newton,khanh2026bundles}. We do not attempt to replace these frameworks. The proposed indices are scalar summaries extracted from a Clarke-type generalized Hessian set. They are best understood as local diagnostics of second-order nonconvexity, not as complete characterizations of variational convexity or algorithmic behaviour.
	
	The contributions are as follows.
	
	\begin{enumerate}[label=(\roman*)]
		\item We define a local interval-valued nonconvexity index for \(C^{1,1}\) functions by evaluating the Davydov--Moldavskaya--Zitikis spectral functional over the generalized Hessian set.
		\item We prove basic consistency and structural properties: reduction to the \(C^2\) case, vanishing for convex \(C^{1,1}\) functions, orthogonal invariance, subadditivity of the upper endpoint under sums, and upper semicontinuity of the upper endpoint.
		\item We relate the upper endpoint to a pointwise curvature modulus that controls weak convexity after a quadratic shift.
		\item We give explicit \(C^{1,1}\setminus C^2\) examples in which the generalized Hessian set and both endpoints of the index can be computed exactly.
		\item We include a cautious smoothing consistency statement for mollifications, while avoiding unsupported claims about general Moreau-envelope convergence of the indices.
	\end{enumerate}
	
	Max-type functions are not treated as a main example in this paper. Although functions of the form \(h=\max\{g_1,g_2\}\) are important in applications, they are typically not \(C^1\), and hence not \(C^{1,1}\), at switching points unless additional compatibility conditions hold. Such functions require either a broader subdifferential framework or a separate piecewise-smooth theory. They are therefore mentioned only as motivation for future work.
	
	\section{The smooth local index}
	
	Let \(\Sd\) denote the space of real \(d\times d\) symmetric matrices and let \(\Spd\subset\Sd\) be the cone of positive semidefinite matrices. For \(Q\in\Sd\), write its spectral decomposition as
	\[
	Q = U \operatorname{diag}(\lambda_1(Q),\ldots,\lambda_d(Q))U^\top .
	\]
	Define
	\[
	\lambda_i^+(Q):=\max\{\lambda_i(Q),0\},
	\qquad
	\lambda_i^-(Q):=\max\{-\lambda_i(Q),0\}.
	\]
	The positive and negative spectral parts of \(Q\) are
	\[
	Q^+
	:=
	U\operatorname{diag}(\lambda_1^+(Q),\ldots,\lambda_d^+(Q))U^\top,
	\]
	and
	\[
	Q^-
	:=
	U\operatorname{diag}(\lambda_1^-(Q),\ldots,\lambda_d^-(Q))U^\top.
	\]
	Thus
	\[
	Q=Q^+-Q^-,
	\qquad
	Q^+\succeq 0,
	\qquad
	Q^-\succeq 0.
	\]
	
	Following \cite{davydov2019}, define the local spectral nonconvexity functional
	\[
	\ell(Q)
	:=
	\|Q^-\|_*
	=
	\sum_{i=1}^d \lambda_i^-(Q),
	\qquad Q\in\Sd,
	\]
	where \(\|\cdot\|_*\) denotes the nuclear norm. If \(h\in C^2(G)\), the smooth local lack-of-convexity index is
	\[
	\LOC_{\mathrm{DMZ}}(h;x)
	:=
	\ell(D^2h(x)).
	\]
	
	The following distance interpretation is useful throughout the paper.
	
	\begin{lemma}[Distance from the positive semidefinite cone]
		\label{lem:distance}
		For every \(Q\in\Sd\),
		\[
		\dist_*(Q,\Spd)
		:=
		\inf_{M\in\Spd}\|Q-M\|_*
		=
		\|Q^-\|_*
		=
		\ell(Q).
		\]
	\end{lemma}
	
	\begin{proof}
		The upper bound follows by taking \(M=Q^+\), since
		\[
		\|Q-Q^+\|_*=\|Q^-\|_*.
		\]
		For the lower bound, let \(P_-\) be the orthogonal projection onto the direct sum of the eigenspaces of \(Q\) corresponding to negative eigenvalues. For any \(M\succeq 0\), the compression \(P_-MP_-\) is positive semidefinite and
		\[
		P_-(Q-M)P_-
		=
		-Q^-|_{\operatorname{range}P_-}
		-
		P_-MP_-.
		\]
		Hence \(P_-(Q-M)P_-\) is negative semidefinite, and therefore
		\[
		\|P_-(Q-M)P_-\|_*
		=
		\tr(Q^-)+\tr(P_-MP_-)
		\geq
		\tr(Q^-)
		=
		\|Q^-\|_*.
		\]
		Since compression by an orthogonal projection cannot increase the nuclear norm,
		\[
		\|Q-M\|_*
		\geq
		\|P_-(Q-M)P_-\|_*
		\geq
		\|Q^-\|_*.
		\]
		Taking the infimum over \(M\succeq0\) gives the desired equality.
	\end{proof}
	
	When \(D^2h(x)\neq 0\), the corresponding normalized smooth index is
	\[
	\NLOC_{\mathrm{DMZ}}(h;x)
	:=
	\frac{\ell(D^2h(x))}{\|D^2h(x)\|_*}.
	\]
	If \(D^2h(x)=0\), we set \(\NLOC_{\mathrm{DMZ}}(h;x)=0\). The smooth convexity index is then
	\[
	\CONV_{\mathrm{DMZ}}(h;x)
	:=
	1-\NLOC_{\mathrm{DMZ}}(h;x).
	\]
	
	\section{Generalized Hessian sets for \(C^{1,1}\) functions}
	
	Let \(G\subset\R^d\) be open and let \(h:G\to\R\) be of class \(C^{1,1}_{\mathrm{loc}}\), meaning that \(h\) is continuously differentiable and \(\nabla h\) is locally Lipschitz. Since \(\nabla h\) is locally Lipschitz, it is differentiable almost everywhere. At points where \(\nabla h\) is differentiable, we write \(D^2h(x)\) for the corresponding symmetric derivative matrix.
	
	\begin{definition}[Generalized Hessian set]
		\label{def:hessian-set}
		Let \(h\in C^{1,1}_{\mathrm{loc}}(G)\) and \(x\in G\). Define the limiting Hessian set
		\[
		\LimHess(h;x)
		:=
		\left\{
		Q\in\Sd:
		\begin{array}{l}
			\exists x_k\to x \text{ such that } \nabla h \text{ is differentiable at }x_k,\\
			D^2h(x_k)\to Q
		\end{array}
		\right\}.
		\]
		The generalized Hessian set of \(h\) at \(x\) is
		\[
		\Hess(h;x)
		:=
		\cl\co \LimHess(h;x).
		\]
	\end{definition}
	
	This is the Clarke generalized Jacobian of the locally Lipschitz mapping \(\nabla h\), restricted to symmetric matrices:
	\[
	\Hess(h;x)=\partial_C(\nabla h)(x).
	\]
	The notation \(\Hess(h;x)\) is used to emphasize the second-order interpretation.
	
	\begin{proposition}[Basic properties]
		\label{prop:hessian-basic}
		Let \(h\in C^{1,1}_{\mathrm{loc}}(G)\) and \(x\in G\). Then:
		\begin{enumerate}[label=(\roman*)]
			\item \(\Hess(h;x)\) is nonempty, compact, and convex.
			\item If \(h\in C^2\) in a neighbourhood of \(x\), then
			\[
			\Hess(h;x)=\{D^2h(x)\}.
			\]
			\item If \(U\) is an orthogonal \(d\times d\) matrix and \(g(z):=h(Uz)\), then
			\[
			\Hess(g;z)
			=
			\{U^\top Q U: Q\in\Hess(h;Uz)\}.
			\]
		\end{enumerate}
	\end{proposition}
	
	\begin{proof}
		Since \(\nabla h\) is locally Lipschitz, Rademacher's theorem implies that \(\nabla h\) is differentiable almost everywhere. On every compact \(K\Subset G\), the local Lipschitz constant of \(\nabla h\) bounds the operator norm of \(D^2h(y)\) at differentiability points \(y\in K\). Hence the set of nearby limiting Hessians is nonempty and locally bounded, and its closed convex hull is nonempty, compact, and convex.
		
		If \(h\in C^2\) near \(x\), continuity of \(D^2h\) gives \(D^2h(x_k)\to D^2h(x)\) for every sequence \(x_k\to x\), proving (ii).
		
		For (iii), at differentiability points of \(\nabla g\),
		\[
		D^2g(z)=U^\top D^2h(Uz)U.
		\]
		Passing to limits and then to closed convex hulls gives the stated transformation rule.
	\end{proof}
	
	\section{The interval-valued local nonconvexity index}
	
	We now define the generalized local index. The key point is that the nonsmooth object is naturally interval-valued.
	
	\begin{definition}[Local nonconvexity interval]
		\label{def:interval-index}
		Let \(h\in C^{1,1}_{\mathrm{loc}}(G)\) and \(x\in G\). Define
		\[
		\underline{\LOC}(h;x)
		:=
		\inf_{Q\in\Hess(h;x)} \ell(Q),
		\]
		and
		\[
		\overline{\LOC}(h;x)
		:=
		\sup_{Q\in\Hess(h;x)} \ell(Q).
		\]
		The local nonconvexity interval of \(h\) at \(x\) is
		\[
		\mathcal I_h(x)
		:=
		\ell(\Hess(h;x))
		=
		\{\ell(Q):Q\in\Hess(h;x)\}.
		\]
	\end{definition}
	
	\begin{proposition}[Interval structure]
		\label{prop:interval}
		For every \(h\in C^{1,1}_{\mathrm{loc}}(G)\) and \(x\in G\),
		\[
		\mathcal I_h(x)
		=
		\left[
		\underline{\LOC}(h;x),
		\overline{\LOC}(h;x)
		\right].
		\]
	\end{proposition}
	
	\begin{proof}
		The set \(\Hess(h;x)\) is compact and convex, hence compact and connected. The functional \(\ell:\Sd\to\R\) is continuous. Therefore \(\ell(\Hess(h;x))\) is a compact connected subset of \(\R\), hence a closed interval. Its endpoints are precisely the infimum and supremum in Definition~\ref{def:interval-index}.
	\end{proof}
	
	\begin{remark}
		The interval-valued formulation is the main conceptual difference from the smooth case. If \(h\) is \(C^2\), then \(\Hess(h;x)\) is a singleton and the interval collapses. At a genuinely nonsmooth \(C^{1,1}\) point, however, different limiting Hessians may display different degrees of negative curvature. The lower endpoint is therefore a best-case local second-order nonconvexity, while the upper endpoint is a worst-case local second-order nonconvexity.
	\end{remark}
	
	\subsection{Normalized indices}
	
	The normalized quantities require a small convention at flat points.
	
	\begin{definition}[Normalized local indices]
		\label{def:normalized}
		Let \(h\in C^{1,1}_{\mathrm{loc}}(G)\) and \(x\in G\). If
		\[
		\Hess(h;x)\setminus\{0\}\neq\varnothing,
		\]
		define
		\[
		\underline{\NLOC}(h;x)
		:=
		\inf_{Q\in\Hess(h;x)\setminus\{0\}}
		\frac{\ell(Q)}{\|Q\|_*},
		\]
		and
		\[
		\overline{\NLOC}(h;x)
		:=
		\sup_{Q\in\Hess(h;x)\setminus\{0\}}
		\frac{\ell(Q)}{\|Q\|_*}.
		\]
		If \(\Hess(h;x)=\{0\}\), set
		\[
		\underline{\NLOC}(h;x)
		=
		\overline{\NLOC}(h;x)
		=
		0.
		\]
		The associated normalized convexity endpoints are
		\[
		\underline{\CONV}(h;x)
		:=
		1-\overline{\NLOC}(h;x),
		\qquad
		\overline{\CONV}(h;x)
		:=
		1-\underline{\NLOC}(h;x).
		\]
	\end{definition}
	
	This convention avoids excluding points where \(0\in\Hess(h;x)\). The zero matrix is omitted from the ratio when nonzero generalized Hessians are present, and the fully flat case is assigned normalized nonconvexity zero.
	
	\section{Basic structural properties}
	
	\subsection{Reduction to the smooth case}
	
	\begin{proposition}[Consistency with the smooth index]
		\label{prop:smooth-reduction}
		If \(h\in C^2\) in a neighbourhood of \(x\), then
		\[
		\underline{\LOC}(h;x)
		=
		\overline{\LOC}(h;x)
		=
		\LOC_{\mathrm{DMZ}}(h;x).
		\]
		If \(D^2h(x)\neq0\), then
		\[
		\underline{\NLOC}(h;x)
		=
		\overline{\NLOC}(h;x)
		=
		\NLOC_{\mathrm{DMZ}}(h;x).
		\]
		If \(D^2h(x)=0\), the normalized indices are all zero.
	\end{proposition}
	
	\begin{proof}
		By Proposition~\ref{prop:hessian-basic},
		\[
		\Hess(h;x)=\{D^2h(x)\}.
		\]
		The infimum and supremum over a singleton are evaluation at the unique element. The normalized statement follows from Definition~\ref{def:normalized}.
	\end{proof}
	
	\subsection{Convex functions}
	
	\begin{proposition}[Convex \(C^{1,1}\) functions]
		\label{prop:convex}
		Let \(h\in C^{1,1}_{\mathrm{loc}}(G)\) be convex. Then, for every \(x\in G\),
		\[
		\Hess(h;x)\subset\Spd.
		\]
		Consequently,
		\[
		\underline{\LOC}(h;x)
		=
		\overline{\LOC}(h;x)
		=
		0,
		\]
		and
		\[
		\underline{\NLOC}(h;x)
		=
		\overline{\NLOC}(h;x)
		=
		0.
		\]
	\end{proposition}
	
	\begin{proof}
		At almost every point where \(\nabla h\) is differentiable, the Hessian \(D^2h\) of a convex \(C^{1,1}\) function is positive semidefinite. Limits of positive semidefinite matrices are positive semidefinite, and convex combinations of positive semidefinite matrices are positive semidefinite. Hence \(\Hess(h;x)\subset\Spd\). Therefore \(\ell(Q)=0\) for every \(Q\in\Hess(h;x)\), proving the result.
	\end{proof}
	
	\begin{remark}
		We do not claim a converse in this paper. Global or neighbourhood-wise convexity from generalized second-order positivity is a more delicate question and belongs to the broader theory of variational convexity and second-order generalized differentiation.
	\end{remark}
	
	\subsection{Orthogonal invariance}
	
	\begin{proposition}[Orthogonal invariance]
		\label{prop:orthogonal}
		Let \(U\) be an orthogonal \(d\times d\) matrix and set \(g(z):=h(Uz)\). Then
		\[
		\underline{\LOC}(g;z)
		=
		\underline{\LOC}(h;Uz),
		\qquad
		\overline{\LOC}(g;z)
		=
		\overline{\LOC}(h;Uz).
		\]
		The same equalities hold for the normalized indices.
	\end{proposition}
	
	\begin{proof}
		By Proposition~\ref{prop:hessian-basic},
		\[
		\Hess(g;z)
		=
		\{U^\top Q U:Q\in\Hess(h;Uz)\}.
		\]
		Orthogonal congruence preserves eigenvalues and nuclear norm. Hence
		\[
		\ell(U^\top Q U)=\ell(Q),
		\qquad
		\|U^\top Q U\|_*=\|Q\|_*.
		\]
		Taking infima and suprema gives the claim.
	\end{proof}
	
	\subsection{Subadditivity under sums}
	
	\begin{proposition}[Upper endpoint subadditivity]
		\label{prop:sum}
		Let \(h_1,h_2\in C^{1,1}_{\mathrm{loc}}(G)\) and set \(h=h_1+h_2\). Then, for every \(x\in G\),
		\[
		\overline{\LOC}(h;x)
		\leq
		\overline{\LOC}(h_1;x)
		+
		\overline{\LOC}(h_2;x).
		\]
	\end{proposition}
	
	\begin{proof}
		Since
		\[
		\nabla h=\nabla h_1+\nabla h_2,
		\]
		the standard Clarke generalized Jacobian sum rule for locally Lipschitz mappings gives
		\[
		\Hess(h;x)
		\subset
		\Hess(h_1;x)+\Hess(h_2;x),
		\]
		where the sum on the right is the Minkowski sum of subsets of \(\Sd\).
		
		Let \(Q\in\Hess(h;x)\). Then \(Q=Q_1+Q_2\) for some
		\[
		Q_i\in\Hess(h_i;x),
		\qquad i=1,2.
		\]
		Using Lemma~\ref{lem:distance},
		\[
		\ell(Q)
		=
		\dist_*(Q,\Spd).
		\]
		For arbitrary \(M_1,M_2\in\Spd\), we have \(M_1+M_2\in\Spd\), and therefore
		\[
		\ell(Q_1+Q_2)
		\leq
		\|Q_1+Q_2-(M_1+M_2)\|_*
		\leq
		\|Q_1-M_1\|_*+\|Q_2-M_2\|_*.
		\]
		Taking infima over \(M_1,M_2\in\Spd\) yields
		\[
		\ell(Q_1+Q_2)
		\leq
		\ell(Q_1)+\ell(Q_2).
		\]
		Thus
		\[
		\ell(Q)
		\leq
		\overline{\LOC}(h_1;x)+\overline{\LOC}(h_2;x).
		\]
		Taking the supremum over \(Q\in\Hess(h;x)\) proves the result.
	\end{proof}
	
	\begin{remark}
		The upper endpoint is the natural one for subadditivity. The lower endpoint need not satisfy an analogous inequality without stronger assumptions, because the minimizing generalized Hessian for \(h_1+h_2\) need not decompose into minimizing generalized Hessians for \(h_1\) and \(h_2\).
	\end{remark}
	
	\subsection{Upper semicontinuity}
	
	\begin{theorem}[Upper semicontinuity of the upper endpoint]
		\label{thm:usc}
		Let \(h\in C^{1,1}_{\mathrm{loc}}(G)\). Then the mapping
		\[
		x\mapsto \overline{\LOC}(h;x)
		\]
		is upper semicontinuous on \(G\).
	\end{theorem}
	
	\begin{proof}
		Let \(x_k\to x\). Passing to a subsequence if necessary, assume that
		\[
		\limsup_{k\to\infty}\overline{\LOC}(h;x_k)
		=
		\lim_{k\to\infty}\overline{\LOC}(h;x_k).
		\]
		For each \(k\), choose \(Q_k\in\Hess(h;x_k)\) such that
		\[
		\ell(Q_k)
		\geq
		\overline{\LOC}(h;x_k)-\frac1k.
		\]
		Choose a compact set \(K\Subset G\) containing \(x\) and all sufficiently large \(x_k\). Since \(\nabla h\) is Lipschitz on \(K\), there is \(L_K>0\) such that
		\[
		\|D^2h(y)\|_2\leq L_K
		\]
		at every differentiability point \(y\in K\). It follows that every matrix in \(\Hess(h;y)\), \(y\in K\), has spectral norm at most \(L_K\). Thus \((Q_k)\) is bounded. Passing to a further subsequence, assume \(Q_k\to Q\).
		
		The Clarke generalized Hessian mapping is outer semicontinuous, since it is the Clarke generalized Jacobian of the locally Lipschitz mapping \(\nabla h\). Therefore
		\[
		Q\in\Hess(h;x).
		\]
		Since \(\ell\) is continuous,
		\[
		\begin{aligned}
			\limsup_{k\to\infty}\overline{\LOC}(h;x_k)
			&=
			\lim_{k\to\infty}\overline{\LOC}(h;x_k)  \\
			&\leq
			\lim_{k\to\infty}\left(\ell(Q_k)+\frac1k\right) \\
			&=
			\ell(Q)
			\leq
			\overline{\LOC}(h;x).
		\end{aligned}
		\]
		This proves upper semicontinuity.
	\end{proof}
	
	\begin{remark}
		The proof uses local Lipschitz boundedness of \(\nabla h\), not outer semicontinuity alone, to obtain compactness of the maximizing sequence. This distinction is important: outer semicontinuity by itself does not automatically supply the boundedness needed in the argument.
	\end{remark}
	
	\section{Relation with weak convexity}
	
	The index \(\overline{\LOC}\) measures the trace-size of the negative spectral part of generalized Hessians. Weak convexity is instead controlled by the most negative eigenvalue. The two quantities are closely related but not identical.
	
	\begin{definition}[Pointwise weak-convexity curvature]
		For \(h\in C^{1,1}_{\mathrm{loc}}(G)\), define
		\[
		\rho(h;x)
		:=
		\sup_{Q\in\Hess(h;x)}
		\max\{0,-\lambda_{\min}(Q)\}.
		\]
	\end{definition}
	
	Thus \(\rho(h;x)\) is the smallest nonnegative number such that
	\[
	Q+\rho(h;x)I\succeq0
	\qquad
	\text{for all }Q\in\Hess(h;x).
	\]
	
	\begin{proposition}[Comparison with the upper endpoint]
		\label{prop:weak-comparison}
		For every \(h\in C^{1,1}_{\mathrm{loc}}(G)\) and \(x\in G\),
		\[
		\rho(h;x)
		\leq
		\overline{\LOC}(h;x)
		\leq
		d\,\rho(h;x).
		\]
	\end{proposition}
	
	\begin{proof}
		For a fixed \(Q\in\Sd\),
		\[
		\max\{0,-\lambda_{\min}(Q)\}
		\leq
		\sum_{i=1}^d \lambda_i^-(Q)
		=
		\ell(Q).
		\]
		Taking suprema over \(Q\in\Hess(h;x)\) gives
		\[
		\rho(h;x)\leq\overline{\LOC}(h;x).
		\]
		Conversely, for every \(Q\in\Sd\),
		\[
		\ell(Q)
		=
		\sum_{i=1}^d \lambda_i^-(Q)
		\leq
		d\,\max\{0,-\lambda_{\min}(Q)\}.
		\]
		Taking suprema again gives
		\[
		\overline{\LOC}(h;x)\leq d\,\rho(h;x).
		\]
	\end{proof}
	
	\begin{proposition}[Local quadratic shift]
		\label{prop:quadratic-shift}
		Let \(U\subset G\) be open and convex. Suppose there exists \(\rho\geq0\) such that
		\[
		Q+\rho I\succeq0
		\qquad
		\text{for all }y\in U,\; Q\in\Hess(h;y).
		\]
		Then
		\[
		y\mapsto h(y)+\frac{\rho}{2}\|y\|^2
		\]
		is convex on \(U\).
	\end{proposition}
	
	\begin{proof}
		Let
		\[
		\widetilde h(y)
		:=
		h(y)+\frac{\rho}{2}\|y\|^2.
		\]
		Then
		\[
		\Hess(\widetilde h;y)
		=
		\Hess(h;y)+\rho I
		\subset
		\Spd
		\]
		for every \(y\in U\). In particular, at almost every point where \(\nabla \widetilde h\) is differentiable,
		\[
		D^2\widetilde h(y)\succeq0.
		\]
		Since \(\widetilde h\in C^{1,1}_{\mathrm{loc}}(U)\), this almost-everywhere positive semidefiniteness of the Hessian implies convexity on the convex set \(U\).
	\end{proof}
	
	\begin{remark}
		Proposition~\ref{prop:weak-comparison} clarifies what the proposed index does and does not measure. The weak-convexity modulus is governed by the largest negative eigenvalue, while \(\overline{\LOC}\) aggregates all negative eigenvalues through the nuclear norm. Thus \(\overline{\LOC}\) is a severity index for total negative curvature, not merely a shift parameter.
	\end{remark}
	
	\section{Examples}
	
	\subsection{A smooth benchmark}
	
	Let
	\[
	h(x,y)=-\cos x-\cos y
	\]
	on \(G=(-a,a)^2\). Then
	\[
	D^2h(x,y)
	=
	\begin{pmatrix}
		\cos x & 0\\
		0 & \cos y
	\end{pmatrix}.
	\]
	Since \(h\in C^2(G)\),
	\[
	\Hess(h;(x,y))=\{D^2h(x,y)\}.
	\]
	Therefore
	\[
	\underline{\LOC}(h;(x,y))
	=
	\overline{\LOC}(h;(x,y))
	=
	(-\cos x)^+ + (-\cos y)^+.
	\]
	This example illustrates only the smooth reduction property.
	
	\subsection{A one-dimensional \(C^{1,1}\setminus C^2\) model}
	
	Let \(a,b\in\R\) and define
	\[
	h_{a,b}(t)
	:=
	\begin{cases}
		\dfrac{a}{2}t^2, & t\geq0,\\[4pt]
		\dfrac{b}{2}t^2, & t<0.
	\end{cases}
	\]
	Then \(h_{a,b}\in C^{1,1}_{\mathrm{loc}}(\R)\), because
	\[
	h_{a,b}'(t)
	=
	\begin{cases}
		at, & t>0,\\
		bt, & t<0,
	\end{cases}
	\qquad
	h_{a,b}'(0)=0,
	\]
	and \(h_{a,b}'\) is locally Lipschitz. Unless \(a=b\), the function is not \(C^2\) at \(0\).
	
	For \(t>0\),
	\[
	h_{a,b}''(t)=a,
	\]
	whereas for \(t<0\),
	\[
	h_{a,b}''(t)=b.
	\]
	Hence
	\[
	\Hess(h_{a,b};0)
	=
	[\min\{a,b\},\max\{a,b\}].
	\]
	Since \(d=1\), the functional is
	\[
	\ell(q)=(-q)^+.
	\]
	Therefore
	\[
	\underline{\LOC}(h_{a,b};0)
	=
	\min_{q\in[\min\{a,b\},\max\{a,b\}]}(-q)^+,
	\]
	and
	\[
	\overline{\LOC}(h_{a,b};0)
	=
	\max_{q\in[\min\{a,b\},\max\{a,b\}]}(-q)^+.
	\]
	
	There are three cases:
	\[
	\begin{array}{lll}
		\text{if } a,b\geq0, & \mathcal I_{h_{a,b}}(0)=\{0\}, & \text{convex case},\\[3pt]
		\text{if } a,b\leq0, &
		\mathcal I_{h_{a,b}}(0)
		=
		[\min\{-a,-b\},\max\{-a,-b\}], &
		\text{fully nonconvex case},\\[3pt]
		\text{if } \min\{a,b\}<0<\max\{a,b\}, &
		\mathcal I_{h_{a,b}}(0)
		=
		[0,\max\{-a,-b\}], &
		\text{mixed case}.
	\end{array}
	\]
	
	\begin{example}[Fully nonconvex kink]
		Let
		\[
		h(t)
		=
		-t^2+\frac12 t|t|.
		\]
		Then
		\[
		h(t)
		=
		\begin{cases}
			-\dfrac12 t^2, & t\geq0,\\[4pt]
			-\dfrac32 t^2, & t<0.
		\end{cases}
		\]
		Thus \(a=-1\) and \(b=-3\), so
		\[
		\Hess(h;0)=[-3,-1].
		\]
		Hence
		\[
		\mathcal I_h(0)=[1,3],
		\]
		that is,
		\[
		\underline{\LOC}(h;0)=1,
		\qquad
		\overline{\LOC}(h;0)=3.
		\]
		For every nonzero \(q\in[-3,-1]\),
		\[
		\frac{\ell(q)}{|q|}=1.
		\]
		Therefore
		\[
		\underline{\NLOC}(h;0)
		=
		\overline{\NLOC}(h;0)
		=
		1.
		\]
	\end{example}
	
	\begin{example}[Mixed generalized curvature]
		Let
		\[
		h(t)=\frac12 t|t|.
		\]
		Then
		\[
		h(t)
		=
		\begin{cases}
			\dfrac12 t^2, & t\geq0,\\[4pt]
			-\dfrac12 t^2, & t<0.
		\end{cases}
		\]
		Thus
		\[
		\Hess(h;0)=[-1,1].
		\]
		Consequently,
		\[
		\mathcal I_h(0)=[0,1].
		\]
		The lower endpoint is zero because the generalized Hessian set contains positive semidefinite matrices; the upper endpoint is one because the same generalized Hessian set also contains negative curvature. This example illustrates why the interval-valued formulation is more informative than a single scalar.
	\end{example}
	
	\begin{table}[h]
		\centering
		\begin{tabular}{lllll}
			\toprule
			Function & Point & \(\Hess(h;x)\) & \(\underline{\LOC}\) & \(\overline{\LOC}\) \\
			\midrule
			\(-\cos x-\cos y\) & \((x,y)\) &
			\(\{\operatorname{diag}(\cos x,\cos y)\}\) &
			\((-\cos x)^++(-\cos y)^+\) &
			same \\
			\(-t^2+\frac12t|t|\) & \(0\) &
			\([-3,-1]\) &
			\(1\) &
			\(3\) \\
			\(\frac12t|t|\) & \(0\) &
			\([-1,1]\) &
			\(0\) &
			\(1\) \\
			\bottomrule
		\end{tabular}
		\caption{Basic examples of the local nonconvexity interval.}
	\end{table}
	
	\section{A cautious smoothing statement}
	
	Smoothing is useful for computation, but a general convergence theory for these indices requires care. The Moreau envelope, for example, plays a central role in weakly convex optimization \cite{davis2019,perezaros2021,denggao2025}, but one should not claim without additional hypotheses that the Hessians of Moreau envelopes converge exactly to the generalized Hessian set used here.
	
	For standard mollification, however, one obtains a useful one-sided consistency statement.
	
	Let \(\varphi\in C_c^\infty(\R^d)\) be a nonnegative mollifier with \(\int_{\R^d}\varphi=1\), and set
	\[
	\varphi_\varepsilon(z)
	:=
	\varepsilon^{-d}\varphi(z/\varepsilon).
	\]
	For \(h\in C^{1,1}_{\mathrm{loc}}(G)\), define
	\[
	h_\varepsilon
	:=
	h*\varphi_\varepsilon
	\]
	on the subset of \(G\) where the convolution is well defined.
	
	\begin{proposition}[Mollification consistency]
		\label{prop:mollification}
		Let \(h\in C^{1,1}_{\mathrm{loc}}(G)\), let \(x_\varepsilon\to x\in G\), and let \(\varepsilon\downarrow0\). Suppose that
		\[
		D^2h_\varepsilon(x_\varepsilon)\to Q
		\]
		along a sequence. Then
		\[
		Q\in\Hess(h;x).
		\]
		Consequently,
		\[
		\ell(Q)\in \mathcal I_h(x).
		\]
	\end{proposition}
	
	\begin{proof}
		Since \(\nabla h\) is locally Lipschitz, its weak derivative \(D^2h\) exists almost everywhere and is locally essentially bounded. For each sufficiently small \(\varepsilon\),
		\[
		D^2h_\varepsilon(x_\varepsilon)
		=
		\int_{\R^d} D^2h(x_\varepsilon-z)\varphi_\varepsilon(z)\,dz,
		\]
		where the integral is understood using the a.e. Hessian. Hence \(D^2h_\varepsilon(x_\varepsilon)\) belongs to the closed convex hull of the essential values of \(D^2h\) in a ball of radius \(O(\varepsilon)\) around \(x_\varepsilon\). Since \(x_\varepsilon\to x\), every cluster point of these convex averages lies in the closed convex hull of limiting Hessians at \(x\), namely \(\Hess(h;x)\). Thus \(Q\in\Hess(h;x)\). The final statement follows from the definition of \(\mathcal I_h(x)\).
	\end{proof}
	
	\begin{remark}
		Proposition~\ref{prop:mollification} is intentionally one-sided. It says that Hessian limits produced by mollification are admissible generalized Hessians. It does not say that every element of \(\Hess(h;x)\) can be recovered from a particular smoothing scheme, nor that the endpoints of the interval converge automatically under smoothing.
	\end{remark}
	
	\section{On possible global indices}
	
	The original smooth framework also considers integrated global indices over a domain. In the \(C^{1,1}\) setting, global constructions are possible but require more bookkeeping than the local theory. The upper endpoint \(x\mapsto\overline{\LOC}(h;x)\) is upper semicontinuous by Theorem~\ref{thm:usc}, hence Borel measurable and locally bounded. Therefore it can be integrated over compact subsets of \(G\).
	
	The lower endpoint and the full interval-valued map require more care. Under only outer semicontinuity of \(x\mapsto\Hess(h;x)\), lower-envelope behaviour need not be as regular. One can impose additional continuity assumptions on the set-valued Hessian map, or work with measurable selections and Aumann-type integrals of the interval-valued map. These issues are not pursued here because they would shift the paper away from its local purpose.
	
	\section{Conclusion}
	
	We have extended the local nonconvexity index of Davydov, Moldavskaya, and Zitikis from \(C^2\) functions to \(C^{1,1}\) functions by replacing the classical Hessian with the Clarke-type generalized Hessian set. The resulting object is naturally interval-valued:
	\[
	\mathcal I_h(x)
	=
	\left[
	\underline{\LOC}(h;x),
	\overline{\LOC}(h;x)
	\right].
	\]
	The lower endpoint records the least visible local second-order nonconvexity and the upper endpoint records the greatest.
	
	The construction is consistent with the smooth theory, vanishes for convex \(C^{1,1}\) functions, respects orthogonal changes of variables, satisfies upper-endpoint subadditivity under sums, and has an upper semicontinuous upper endpoint. The upper endpoint also controls, up to dimension-dependent constants, a pointwise weak-convexity curvature modulus. Explicit one-dimensional examples show that the interval formulation is not cosmetic: at a nonsmooth \(C^{1,1}\) point, the generalized Hessian set may contain both convex and nonconvex limiting curvatures.
	
	The paper deliberately avoids broader claims. It does not characterize variational convexity, does not handle arbitrary max-type nonsmooth functions, and does not provide a convergence theory for generalized Newton or proximal algorithms. Its contribution is a local scalar and interval-valued diagnostic that may complement, but not replace, the richer second-order machinery of modern variational analysis.
	
	Future work could develop integrated global versions of the interval index, investigate stability under perturbations of \(h\), and study whether the upper endpoint correlates with algorithmic difficulty in weakly convex or prox-regular optimization problems.

\end{document}